\documentclass[12pt]{amsart}
\usepackage{amssymb, amsmath}
\usepackage{bm}
\usepackage{float,epsfig}
\usepackage{epsfig,amsfonts,amsbsy,amssymb,amsthm,here}
\usepackage[usenames,dvipsnames]{color}
\usepackage{graphics}
\usepackage{graphicx}
\usepackage{epsfig}
\usepackage{url}
\usepackage{arydshln,leftidx,mathtools}
\usepackage{algpseudocode}
\usepackage{algorithm}
\usepackage{mathtools}
\usepackage{enumitem}
\usepackage{comment}
\usepackage{todonotes}

\newtheorem{theorem}{\bf Theorem}[section]
\newtheorem{proposition}[theorem]{\bf Proposition}
\newtheorem{definition1}[theorem]{\bf Definition}
\newtheorem{corollary}[theorem]{\bf Corollary}
\newtheorem{conjecture}[theorem]{\bf Conjecture}
\newtheorem{example1}[theorem]{\bf Example}

\newtheorem{remark1}[theorem]{\bf Remark}
\newtheorem{lemma}[theorem]{\bf Lemma}

\newenvironment{remark}{\begin{remark1}\rm}{\end{remark1}}
\newenvironment{example}{\begin{example1}\rm}{\end{example1}}
\newenvironment{definition}{\begin{definition1}\rm}{\end{definition1}}

\newcommand{\ba}{\begin{array}}
\newcommand{\ea}{\end{array}}

\def \R{{\mathbb R}}
\def \C{{\mathbb C}}
\def \Z{{\mathbb Z}}

\def \x{{\mathbf{x}}}
\def \vv{{\mathbf{v}}}
\def \w{{\mathbf{w}}}

\def \e{{\mathbf{e}}}

\def \S{{\mathcal S}}

\def \Delta{\triangle}

\def \diag{\mathrm{diag}}

\def \supp{\mathrm{supp}}
\def \diag{\mathrm{Diag}}

\def \Re{\mathsf{Re}}
\def \Im{\mathsf{Im}}

\def \sup{\mathrm{sup}}

\def \pos{\mathrm{pos}}
\DeclareMathOperator{\relint}{relint}

\def \x{{\mathbf{x}}}

\def \z{{\mathbf{z}}}

\def \w{{\mathbf{w}}}

\def \a{{\mathbf{a}}}

\def \e{{\mathbf{e}}}

\DeclareMathOperator{\adj}{adj} 

\DeclareMathOperator{\inter}{int} 
\DeclareMathOperator{\GL}{GL} 
\DeclareMathOperator{\init}{in}

\DeclareMathOperator{\initial}{in} 

\newcommand{\sym}{\mathcal{S}}

\title{Combinatorics and preservation of conically stable polynomials}

\author{Giulia Codenotti}

\address{Giulia Codenotti: Freie Universit\"at Berlin, Arnimallee 3,
  14105 Berlin, Germany}

\email{giulia.codenotti@fu-berlin.de}

\author{Stephan Gardoll}

\author{Thorsten Theobald}

\address{Stephan Gardoll, Thorsten Theobald:
  Goethe-Universit\"at, FB 12 -- Institut f\"ur Mathematik,
  Postfach 11 19 32, 60054 Frankfurt am Main, Germany}

\email{\{gardoll,theobald\}@math.uni-frankfurt.de}

\thanks{An extended abstract of this work was accepted for
presentation at the Workshop Discrete Mathematics Days 2022, 
Santander.}

\date{\today}

\begin{document}

\begin{abstract}
Given a closed, convex cone $K\subseteq \mathbb{R}^n$, a multivariate polynomial $f\in\mathbb{C}[\mathbf{z}]$ is called $K$-stable if the imaginary parts of its roots are not contained in the relative interior of $K$. If $K$ is the non-negative orthant, $K$-stability specializes to the usual notion of stability of polynomials.
 
We develop generalizations of preservation operations and of combinatorial criteria from usual stability towards conic stability. 
A particular focus is on the cone of positive semidefinite matrices
(psd-stability). In particular, we prove the preservation of psd-stability
under a natural generalization of the inversion operator.
Moreover, we give conditions on the support of psd-stable 
polynomials and
characterize the support of special families of psd-stable polynomials.
\end{abstract}

\maketitle

\section{Introduction\label{se:introduction}}

Multivariate stable polynomials can be seen as a generalization of 
real-rooted polynomials,
and they enjoy many connections to other branches in mathematics,
including differential equations \cite{borcea-braenden-2010},
optimization \cite{straszak-vishnoi-2017},
probability theory \cite{bbl-2009}, 
matroid theory \cite{braenden-hpp,cos-2004},
applied algebraic geometry \cite{volcic-2019},
theoretical computer science \cite{mss-interlacing1, mss-interlacing2}
and statistical physics \cite{borcea-braenden-leeyang1}. 
See also the surveys of
Pemantle \cite{pemantle-2012} and Wagner \cite{wagner-2010}.

Classical related notions include \emph{hyperbolic polynomials} \cite{garding-59} 
or stability with respect to an arbitrary domain (see, e.g.,
\cite{gkv-2016} and the references therein).
Recently, further variants and generalizations 
have been developed, including 
\emph{conic stability} introduced by J\"orgens and the third 
author \cite{joergens-theobald-conic},
\emph{Lorentzian polynomials} introduced 
by Br\"and\'{e}n and Huh \cite{lorentzian-braenden-huh}
and \emph{positively hyperbolic varieties} introduced by Rinc\'{o}n, 
Vinzant and Yu \cite{rincon-vinzant-yu}.

In this work we focus on the notion of conic stability.
Given a closed, convex cone $K\subseteq \R^n$, a polynomial $f\in\C[\z]=\C[z_1,\ldots,z_n]$ is called \emph{$K$-stable}, if $\Im(\z)\not\in\relint K$ for every root $\z$ of $f$, 
where $\Im(\z)$ denotes the vector of the imaginary parts of the
components of $\z$ and $\relint K$ denotes the relative interior of $K$.
Note that $(\R_{\geq 0})^n$-stability coincides with the usual stability. In the case of a homogeneous polynomial, $K$-stability of $f$ is 
equivalent to the containment of $\relint K$ in a hyperbolicity cone of $f$. The notion of $K$-Lorentzian polynomials recently introduced by Br\"and\'{e}n and Leake \cite{br-le-21} is, up to scaling, a generalization of homogeneous $K$-stable polynomials. 
Stability with respect to the positive semidefinite cone on the
space of symmetric matrices is denoted as \emph{psd-stability}.
In the homogeneous case such polynomials are also known as
Dirichlet-G{\aa}rding polynomials \cite{harvey-lawson-2009}.
Prominent subclasses of psd-stable 
polynomials arise from determinantal 
representations \cite{dgt-2021}. Blekherman, Kummer, Sanyal
et al.\ \cite{blek-et-al-2021} 
have constructed a family of psd-stable \emph{lpm-polynomials}
(\emph{linear principle minor polynomials}) from multiaffine stable polynomials.

The purpose of the current paper is to initiate the study of generalizations 
of two prominent research directions in stable polynomials towards
conically stable polynomials: preservation operators and
combinatorial criteria. In particular, a focus is to understand
the transition from the classical stability situation 
to the conic stability with respect to 
non-polyhedral cones such as the positive semidefinite cone.

With regard to preservation, stable polynomials
have been recognized to remain stable under a number of operations,
see the survey \cite{wagner-2010}. Prominent examples include
the inversion operation (see \cite{borcea-braenden-leeyang1}),
the preservation under taking partial derivatives (as a consequence
of the univariate Gau{\ss}-Lucas Theorem),
the Lieb-Sokal Lemma (\cite[Lemma 2.3]{lieb-sokal-1981}, see
also \cite[Lemma 2.1]{borcea-braenden-leeyang1})
and the celebrated characterization of Borcea and Br\"and\'{e}n 
of linear operators preserving 
stability \cite[Theorem 1.3]{borcea-braenden-leeyang1}.
Many of the mentioned applications of stability rely on the 
preservation properties.

With regard to \emph{combinatorial criteria}, 
several important combinatorial results have been achieved, 
which provide effective criteria 
for the recognition of stable polynomials. A groundbreaking result
of Choe, Oxley, Sokal and Wagner states 
that the support of a multi-affine, homogeneous and stable 
polynomial $f \in \R[\mathbf{z}] = \R[z_1, \ldots, z_n]$ 
is the set of bases of a matroid
\cite[Theorem 7.1]{cos-2004}.
Br\"and\'{e}n \cite[Theorem 3.2]{braenden-hpp} proved a generalization of this result for the support of any stable polynomial $f \in \R[\mathbf{z}]$, showing that it forms a jump system, i.e., it satisfies the so-called 
Two-Steps Axiom. See Section~\ref{se:prelim} for formal definitions.
Recently, Rinc\'{o}n, Vinzant
and Yu gave an alternative proof of the matroid result, based
on a tropical proof of the auxiliary statement that positive
hyperbolicity of a variety is preserved under passing 
over to the initial form \cite[Corollary 4.9]{rincon-vinzant-yu}.

The proofs of these combinatorial properties strongly rely on the
preservation properties of stable polynomials. These
preservation properties establish the connection between
the combinatorial and the algebraic viewpoint. For example,
taking the partial derivative of a polynomial $f$ shifts the
support vectors of $f$ by a unit vector in a negative coordinate
direction (and some support vectors may disappear). Since
stability of a polynomial is preserved under taking partial
derivatives, one can use this preserver to argue about the
combinatorics of the support. In the univariate case, these
considerations are classical for deriving log-concavity of sequences
with real-rooted generating functions.

\smallskip

\noindent
{\bf Our contributions.}
1. We generalize several preserving operators for usual stability
  to the conic stability. In particular, we derive a conic version
  of the Lieb-Sokal Lemma (see Lemma~\ref{le:lieb-sokal}
  and Corollary~\ref{co:ref-lieb-sokal}).

2. For the case of psd-stability,
  we can prove the preservation under a natural generalization
  of the inversion operator. See Theorem \ref{th:psd-inversion}.
  This generalized inversion operator
  is specific to the case of psd-stability and exhibits a prominent
  role of this class. Furthermore, we show that psd-stable polynomials are preserved under taking initial forms with respect to positive definite matrices. See Theorem~\ref{th:initial-form-W-pd}.

3. \emph{Combinatorics of psd-stable polynomials.}
 We prove a necessary criterion on the support of any psd-stable polynomial
 in Theorem~\ref{th:struc-thm} and characterize the support of special
 families of psd-stable polynomials. In particular, we characterize psd-stability
 of binomials (Theorem~\ref{th:bin-class-dia}), give a necessary 
 criterion for psd-stability of a larger class containing binomials
 (Theorem~\ref{th:hom-non-mix}),
 and introduce a class of polynomials of determinants, which satisfies a
  generalized jump system criterion with regard to psd-stability.
  Theorem~\ref{th:mat-size-poly-of-det} characterizes the restrictive
  structure of psd-stable polynomials of determinants.
  These results are complemented by an additional conjecture on the
  support of general psd-stable polynomials. We provide evidence
  for this conjecture by verifying it for the classes of polynomials treated previously.
  
\medskip

The paper is structured as follows. 
Section~\ref{se:prelim} collects relevant background on preservers of the usual stability notion as well as an introduction to the notion of $K$-stability.

In Section~\ref{se:preservers}, we study preservers of conic stability for 
general and polyhedral cones, including the generalized
version of the Lieb-Sokal Lemma. 
Section~\ref{se:psd-preservers} treats the case of psd-stability,
in particular, the preservation of psd-stability under an inversion operation
and under passing over to certain initial forms.
Section~\ref{se:combinatorics} deals with combinatorial
conditions of psd-stable polynomials.
Therein, Subsections~\ref{se:binomials-non-mixed} 
and~\ref{sse:pol-of-det} discuss the support of
special families of psd-stable polynomials. 
Subsection~\ref{sse:comb-gen-psd} considers the support of general
psd-stable polynomials and also raises a conjecture.

\subsection*{Acknowledgment}
This work was supported through DFG grant TH 1333/7-1. The research was
primarily done while the first author was with Goethe University, supported
by an Early Career Researcher grant of Goethe University.

\section{Preliminaries}\label{se:prelim}

Let $\R_{\geq 0}$ and $\R_{>0}$ denote the sets of non-negative and of
positive real numbers. 
Further, let
$\mathcal{H} := \{z \in \C \, : \, \Im(z) > 0\}$ be the 
\emph{open upper half-plane} of $\C$.
Throughout the text,
bold letters will denote $n$-dimensional vectors unless noted otherwise.

In this section, we collect known properties of stable polynomials and then introduce the generalization of stability, namely conic stability, with which
the paper is concerned. 

\subsection{Stable polynomials}

A polynomial $f\in\C[\z]$ is called \emph{stable} if 
for every root $\z$ of $f$, there exists some $j \in [n]$
with $\Im(z_j)\leq 0$. Hence, a univariate real polynomial $f$ is stable if and
only if it is real-rooted, because the non-real roots of univariate real
polynomials occur in conjugate pairs.
The following collection from \cite[Lemma 2.4]{wagner-2010}
recalls some elementary operations that preserve stability, where
f) can be derived from the Gau{\ss}-Lucas Theorem. Denote by $\deg_i$ the
degree in the variable $z_i$.

\begin{proposition}\label{le:usual-stab-pres}
Let $f \in \C[\mathbf{z}]$ be stable. \\
$\textnormal{a) Permutation}$: $f(z_{\sigma(1)},\ldots,z_{\sigma(n)})$ is stable for every permutation $\sigma:[n]\rightarrow [n]$. \\
$\textnormal{b) Scaling}$: $c\cdot f(a_1z_1,\ldots,a_nz_n)$ is stable or zero for every $c\in\C$ and $\a\in\R^n_{>0}$. \\
$\textnormal{c) Diagonalization}$: $f(\z)\,\rule[-2mm]{0.1mm}{5mm}_{z_j=z_i}\in\C[z_1,\ldots,z_{j-1},z_{j+1}, \ldots,z_n]$ is stable or zero for every $i \neq j\in[n]$. \\
$\textnormal{d) Specialization}$:  $f(b,z_2,\ldots,z_n)\in\C[z_2,\ldots,z_n]$ is stable or zero for every $b\in\C$ with $\Im(b)\geq 0$. \\
$\textnormal{e) Inversion}$: $z_1^{\deg_1(f)}\cdot f(-z_1^{-1},z_2,\ldots,z_n)$ is stable. \\
$\textnormal{f) Differentiation}$: $\partial_j f(\z)$ is stable
or zero for every $j\in [n]$.
\end{proposition}

A prominent linear stability preserver is the 
Lieb-Sokal Lemma (\cite[Lemma 2.3]{lieb-sokal-1981}, 
see also \cite[Lemma 2.1]{borcea-braenden-leeyang1} or 
\cite[Lemma 3.2]{wagner-2010}). It is an essential ingredient in
Borcea and Br\"and\'{e}n's full characterization of linear operations preserving
stability \cite[Theorem 1.1]{borcea-braenden-leeyang1}, see also
\cite[Section 3.2]{borcea-braenden-2010}. 

\begin{proposition}[Lieb-Sokal Lemma]\label{le:lieb-sokal}
	Let $g(\z)+yf(\z)\in\C[\z,y]$ be stable and assume $\deg_i(f)\leq 1$. Then $g(\z)-\partial_if(\z)\in\C[\z]$ is stable or identically zero.
\end{proposition}

The following statement due to Hurwitz allows us to obtain (conic) stability statements
as limit of statements on compact subsets under a uniform convergence
condition.

\begin{proposition}\cite[Par.\ 5.3.4]{Hurwitz-Thm-1999}\label{th:hurwitz}
Let $\{f_k\}$ be a sequence of polynomials non-vanishing in a connected open set $U \subseteq \R^n$, and assume it converges to
a function $f$ uniformly on compact subsets of U. Then $f$ is either non-vanishing on $U$ or it is identically zero.
\end{proposition}

As a consequence of~\cite[Theorem 6.1]{cos-2004}, the following 
necessary condition for homogeneous stable polynomials based on their coefficients applies.
\begin{theorem}
  \label{th:same-phase}
  All nonzero coefficients of a homogeneous stable polynomial $f\in\C[\z]$ 
  have the same phase.
\end{theorem}

\subsection{Stability and initial forms}

The initial form $\initial_\w(f)$ of a polynomial $f(\z)=\sum_{\alpha \in S} c_\alpha \z^\alpha$ with respect to a functional $\w$
in the dual space $(\R^n)^*$ is defined as
\[
\initial_\w(f)= \sum_{\alpha \in S_\w} c_\alpha \z^\alpha,
\]
where $S_\w:=\{\alpha \in S \, : \, \langle \w, \alpha\rangle = \max_{\beta \in S} \langle \w, \beta\rangle \}$ and 
$\langle \cdot, \cdot \rangle$ is the natural dual pairing.
That is, we restrict the polynomial $f$ to those monomials whose exponents lie on the face of the Newton polytope of $f$ where the functional 
$\w$ is maximized. 

In the context of their work on positively hyperbolic varieties,
Rinc\'{o}n, Vinzant and Yu \cite[Proposition 4.1]{rincon-vinzant-yu}
showed that for polynomials with real coefficients, 
stability is preserved under taking initial forms.
Their proof is based on tropical geometry.
For the convenience of the reader, we give here a simplified proof,
and at the same time slightly generalize 
the statement to also cover polynomials with complex coefficients.
The observation that the statement is also valid for complex coefficients
has independently been derived by Kummer and Sert 
\cite[Proposition 2.6]{kummer-sert-2021}. 

\begin{theorem}\label{th:stable-initial-form}
	If $f\in\C[\z]$ is stable and $\w \in (\R^n)^* 
	\setminus \{ \mathbf{0} \}$, 
	then $\init_\w(f)(\z)$ is also stable.
\end{theorem}

\begin{proof}
	Let $\varphi:=\max\left\{\langle \alpha,\w\rangle: \alpha\in\text{supp}(f)\right\}$, and for $\lambda > 0$, define the polynomial
	$f_\lambda(\z):=\frac{1}{\lambda^\varphi}\cdot 
	f(\lambda^{w_1}z_1,$ $ \ldots, $ $\lambda^{w_n}z_n)$, which is stable by Proposition~\ref{le:usual-stab-pres}.
	
To apply Hurwitz' Theorem to finally achieve stability of the initial form, we need to ensure that $f_\lambda$ converges uniformly to $\init_\w(f)$ on every compact subset $C\subseteq \C^n$. Let $\mu=\max\{\langle \alpha,\w\rangle:\langle \alpha,\w \rangle < \varphi, \alpha\in\supp(f)\}$ and $\delta=\varphi-\mu>0$. Then
	\begin{align*}
		 \lim\limits_{\lambda\rightarrow \infty}\underset{\z\in C}{\sup} 
		 \left| f_\lambda(\z) - \init_\w(f)(\z) \right|	
		& \leq  
		 \lim\limits_{\lambda\rightarrow \infty}\underset{\z\in C}{\sup} 
		 \sum\nolimits_{\langle \alpha,\w\rangle <\varphi} \left| \frac{1}{\lambda^{\delta}} c_\alpha \z^\alpha \right| \\ 
		& = \ 
		\lim\limits_{\lambda\rightarrow \infty} \frac{1}{\lambda^{\delta}} \underset{\z\in C}{\sup} \sum\nolimits_{\langle \alpha,\w\rangle <\varphi} \left|  c_\alpha \z^\alpha \right| \ = \ 0,
	\end{align*}
	since the norm in the last equality is bounded, given that $C$ is a compact set. 
\end{proof}

The discussion of the preservation of conically stable polynomials when passing over to initial forms is continued at the end of 
Section~\ref{se:psd-preservers}.

\subsection{Combinatorics of stable polynomials}\label{sec:comb-stable}

For $\alpha,\beta\in\Z^n$, the \emph{steps} 
between $\alpha$ and $\beta$ are defined as the set
\begin{equation*}
\text{St}(\alpha,\beta):=\left\{\sigma\in\Z^n: |\sigma|=1, |\alpha+\sigma-\beta|=|\alpha-\beta|-1 \right\},
\end{equation*}
where $|\sigma| := \sum_{i=1}^n |\sigma_i|$.
A collection of points $\mathcal{F}\subseteq \Z^n$ is called a 
\emph{jump system} if for every $\alpha,\beta\in\mathcal{F}$ and 
$\sigma\in\text{St}(\alpha,\beta)$ with $\alpha+\sigma\notin\mathcal{F}$ there is some $\tau\in\text{St}(\alpha+\sigma,\beta)$ such that $\alpha+\sigma+\tau\in\mathcal{F}$. In words, if after one step from $\alpha$ towards $\beta$ we have left the set $\mathcal{F}$, then there must be a second step that takes us back into $\mathcal{F}$. This property is also
known as the \emph{Two-Steps Axiom}. The \emph{support} of a complex polynomial $f(\z)=\sum_{\alpha}c_\alpha \z^\alpha$ is defined as $\text{supp}(f)=\{\alpha\in\Z^n_{\geq 0}:c_\alpha\neq 0\}$, that is, it is the set of all exponent vectors $\alpha$ such that the corresponding coefficient $c_\alpha$ is non-zero in $f$. The following theorem reveals the connection between stable polynomials and jump systems.
\begin{theorem}[Br\"and\'{e}n \cite{braenden-hpp}]\label{th:braenden-js}
	If $f\in\C[\z]$ is stable, then its support  
	is a jump system.
\end{theorem}

In \cite[Proposition 4.1]{rincon-vinzant-yu}, the support of stable binomials is explicitly classified as follows.
Here, $\e_i$ denotes the $i$-th unit vector in $\R^n$.
\begin{theorem}\label{th:binom-class-stable}
	Let $f=c_\alpha\z^\alpha+c_\beta\z^\beta$ with $c_\alpha,c_\beta\neq 0$ and $\alpha,\beta \in \Z_{\geq 0}^n$ be stable
	and let $\z^{\alpha}$ and $\z^{\beta}$ do not have a common
	factor.
	 Then one of the following holds,
	\begin{enumerate}[label=\alph*)]
		\item $\{\alpha,\beta\}=\{0,\e_i\}$ for some $i\in[n]$,
		\item $\{\alpha,\beta\}=\{\e_i,\e_j\}$ for some $i,j\in[n]$ and $\frac{c_\alpha}{c_\beta}\in\R_{\geq 0}$, or
		\item $\{\alpha,\beta\}=\{0,\e_i+\e_j\}$ for some $i,j\in[n]$ and $\frac{c_\alpha}{c_\beta}\in \R_{<0}$. 
	\end{enumerate}
\end{theorem}
\subsection{Conic stability}
The following notion of conic stability as introduced in \cite{joergens-theobald-conic} generalizes stability to more general cones.
Let $K$ be a closed, convex cone in $\R^n$ and denote by 
$\relint K$ its relative interior.

\begin{definition}\label{def:K-stable}
	A polynomial $f \in \C[\mathbf{z}]$ is called $K$-\emph{stable}, if $f(\mathbf{z})\neq0$ whenever $\Im(\mathbf{z})\in \relint K$.
\end{definition}

Observe that by choosing the cone $K = \R^n_{\ge 0}$, we recover the usual notion of stability. 
For any closed, convex cone $K$, conic stability can be characterized through stability of univariate polynomials 
(see \cite[Lemma 3.4]{joergens-theobald-conic}, that proof
literally also works without the assumption of full-dimensionality
made there).

\begin{proposition}[\cite{joergens-theobald-conic}, Lemma 3.4]
\label{pr:reduce-to-univariate}
A polynomial $f \in \C[\mathbf{z}] \setminus \{0\}$ is $K$-stable if and only
if for all $\mathbf{x}, \mathbf{y} \in \R^n$ with $\mathbf{y} \in \relint K$,
the univariate polynomial $t \mapsto f(\mathbf{x} + t\mathbf{y})$ is stable
or identically zero.
\end{proposition}

\begin{remark}
\label{re:real-coeff}
A homogeneous polynomial $f\in\C[\z]$ is called \emph{hyperbolic} w.r.t. 
$\e\in\R^n$ if $f(\e)\neq 0$ and the univariate polynomial $t\mapsto f(\x+t\e)$ is real rooted. For a full-dimensional cone $K \subset \R^n$,
every homogeneous $K$-stable polynomial is hyperbolic w.r.t.\ every 
$\e\in \relint K = \inter K$ by \cite[Theorem 3.5]{joergens-theobald-conic}
and hence, up to a multiplicative constant every
homogeneous $K$-stable polynomial has real coefficients \cite{garding-59}.
\end{remark}

\subsection{Positive semidefinite stability}\label{sec:psd}
We introduce the notion of psd-stability, an important special case of conic stability where the cone is chosen to be the positive semidefinite cone.

Denote by $\sym^\C_n$ the vector space of complex symmetric matrices 
(rather than Hermitian matrices)
and by $\sym_n$ the space of real ones. The cones of real positive semidefinite and positive definite matrices are denoted by $\sym^{+}_n$ and $\sym^{++}_n$.
Let $\C[Z]$ denotes the ring of polynomials on the symmetric
matrix variables $Z = (z_{ij})$. More precisely, $\C[Z]$ is the 
vector space generated by monomials of the form
$Z^{\alpha} = \prod_{1 \le i,j \le n} z_{ij}^{\alpha_{ij}}$ with some
non-negative symmetric matrix $\alpha$ whose diagonal entries are integers
and whose off-diagonal entries are half-integers. Polynomials
in $\C[Z]$ can also be interpreted as polynomials in the polynomial
ring  $\C[\{z_{ij}| 1\leq i\leq j \leq n\}]$, by identifying
$z_{ij}$ and $z_{ji}$ for $i \neq j$. For example, consider the
monomial
\[
  Z^{{\scriptstyle \begin{pmatrix}
                                0 & 1/2 \\
                                1/2 & 0
                                \end{pmatrix}}} \ = \ z_{12}^{1/2} z_{21}^{1/2}
                                \ = \ z_{12}
\]
in the polynomial ring $\C[Z]$ over the vector space $\sym_2^{\C}$.

\begin{definition}
\emph{Psd-stability} is defined as $\sym_n^{+}$-stability for polynomials over the vector space $\sym^\C_n$ of complex symmetric matrices. That is, a polynomial $f \in \C[Z]$
is psd-stable if it has no root $M \in \sym^\C_n$ such that $\Im(M) \in \sym^{++}_n$.
\end{definition}

The \emph{support} $\supp(f)$ of a polynomial $f \in \C[Z]$
is the set of all symmetric exponent matrices of the monomials occurring with non-zero coefficients in the polynomial.
The variables $z_{ii}$ are called \emph{diagonal variables}, while the variables $z_{ij}$ with $i \neq j$ are the \emph{off-diagonal variables}.
We say that a monomial with exponent matrix $\alpha$ is a \emph{diagonal monomial} if $\alpha_{ij}=0$ for all $i\neq j \in [n] $, and we say that it is an \emph{off-diagonal} monomial if $\alpha_{ii}=0$ for all $i \in [n]$. By convention, we say that a constant is a diagonal monomial, but not an off-diagonal one.

\begin{example}
        Let $f(Z)=\det(Z)$ in the polynomial ring $\C[Z]$ over the vector
        space $\sym_2^{\C}$. Then
        \[
                f(Z) \ = \ z_{11}z_{22}-z_{12}^2 \ = \ Z^{{\scriptstyle \begin{pmatrix}
                                1 & 0 \\
                                0 & 1
                                \end{pmatrix}}} - Z^{{\scriptstyle \begin{pmatrix}
                                0 & 1 \\
                                1 & 0
                                \end{pmatrix}}}.
        \]
The monomial $z_{11}z_{22}$ is a diagonal monomial while the other one is
an an off-diagonal monomial.
\end{example}

A prime example of psd-stable polynomials are determinants. The proof is included for completeness. 

\begin{lemma}\label{le:determinant}
$f(Z)=\det(Z)$ is psd-stable.
\end{lemma}
\begin{proof}
Suppose that $f$ is not psd-stable, that is, there exist real symmetric matrices $A$ and $B$ with $B$ positive definite, such that $f(A+i B)=0$. Then $B$ is invertible and $0=f(A +i B) = \det(A+iB)  = \det(B) \det(B^{-\frac{1}{2}}AB^{-\frac{1}{2}}+ i I_n)$, where $I_n$ denotes the identity matrix
of size $n$.
Hence, $-i$ is a root of the characteristic polynomial of,
and thus an eigenvalue of, $B^{-\frac{1}{2}}AB^{-\frac{1}{2}}$: a contradiction, since a symmetric real matrix has only real eigenvalues.
\end{proof}

        Contrary to the usual stability notion, monomials are not necessarily psd-stable. In fact, every monomial with an off-diagonal variable as a factor, is not psd-stable since it evaluates to zero for $Z=i\cdot I_n$.

Psd-stability can be viewed as stability with respect
to the \emph{Siegel upper half-space}
$\mathcal{H}_\mathcal{S} \ =  \{ A \in \C^{n \times n} \text{ symmetric } \, : \, \Im(A) \text{ is positive definite}\}$.
The Siegel upper half-space occurs in algebraic
geometry and number theory as the domain of modular forms.

\section{Preservers for conic stability}\label{se:preservers}

We provide generalizations of the stability preservers
from Section \ref{se:prelim} to conic stability with respect to some
closed, convex cone $K$. Our focus is on general 
cones and on
the subclass of polyhedral cones. A main
result in this section is Theorem~\ref{le:conic-lieb-sokal}, a conic version of the Lieb-Sokal Lemma. In Section~\ref{se:psd-preservers}, the specific case of preservers for psd-stability will be studied.

A conical analogue of property b) 
from Proposition~\ref{le:usual-stab-pres}, scaling, holds trivially since $K$ is a cone: for any $c\in\C$ and $a\in\R_{\geq 0}$, the polynomial
$c\cdot f(az_1,\ldots,az_n)$ is $K$-stable or identically zero.
We now study the preservation of conical stability under directional derivatives. For a vector 
$\vv \in \R^n \setminus \{\mathbf{0}\}$, 
denote by $\partial_{\vv}$ the directional
derivative in direction $\vv$, i.e.,
$\partial_\vv f(\z) = \frac{d}{dt}f(\z + t\vv) \big|_{t=0}$.

\begin{lemma}\label{le:conic-deriv}
Let $f\in\C[\z]$ be $K$-stable. For $\vv \in K$, the polynomial $\partial_\vv f$ is $K$-stable or identically zero.

\end{lemma}

In the homogeneous case, this statement follows from the concept
of a Renegar derivative \cite{renegar} for hyperbolic polynomials.

\begin{proof}
Let $f$ be $K$-stable and $\vv\in K$. 
Assume that $\partial_\vv f$ is neither $0$ nor $K$-stable. Then
there is some $\z\in\C^n$ such that $\Im(\z)\in\relint K$ 
and $\partial_\vv f(\z)=0$. 

To aim at a contradiction to the univariate Gau{\ss}-Lucas Theorem, we
construct through a substitution in $f$ a
univariate polynomial $g \not\equiv 0$, which has a non-real zero.
Since $\Im(\z) \in \relint(K)$, there exists some $\varepsilon > 0$ such that
$\Im(\z) - \varepsilon \vv \in \relint K$.
Define the univariate polynomial $g : t\mapsto f(\z-i \varepsilon \vv+t\vv)$.
If $g \equiv 0$, then $f(\z) = g(i \varepsilon) = 0$ in contradiction to the $K$-stability
of $f$. Hence, $g \not\equiv 0$.
Since $\Im(\z) - \varepsilon \vv \in \relint(K)$ and $\vv \in K$, the univariate polynomial $g$ is stable: if it had any root $t$ with $\Im(t)>0$, $\z-i\varepsilon \vv +t\vv$ would be a root of $f$, but its imaginary part $\Im(\z)-\varepsilon \vv +\Im(t) \vv$ is in the relative interior of the cone $K$, a contradiction to the conical stability of $f$.
Moreover, $g$ is not constant, because $\partial_\vv f \not\equiv 0$.
Hence, by the Gau{\ss}-Lucas Theorem, the derivative $g'$ is stable.
Since
	\begin{equation*}
	g'(i \varepsilon)=\frac{\partial}{\partial t}f(\z-i \varepsilon \vv+t\vv)\Big|_{t=i \varepsilon}
	=\partial_\vv f(\z)=0,
	\end{equation*} 
we obtain a contradiction to the stability of $g'$.
\end{proof}

There is a natural generalization of property d) in~Lemma~\ref{le:usual-stab-pres} to conic stability. 

\begin{lemma}
	\label{le:cone-subset}
	Let $f \in \C[\mathbf{z}]$ be $K$-stable, $\a \in \C^n$ and
	$\mathbf{v}^{(1)}, \ldots, \mathbf{v}^{(k)} \in \R^n$.
	Further set $K' = \pos \{\mathbf{v}^{(1)}, \ldots, \mathbf{v}^{(k)} \}$ 
	and assume that
	$\Im(\mathbf{a}) + K' \subseteq K$.
	Then the polynomial $g \in \C[\mathbf{z}]$ defined by
	\[
	g(z_1, \ldots, z_{k}) \ = \ f(\mathbf{a} + 
	\sum_{j=1}^{k} z_j \mathbf{v}^{(j)})
	\]
	is stable or the zero polynomial.
\end{lemma}

Setting $K = \R^n_{\ge 0}$, $k=n-1$, $\vv^{(j)} = \mathbf{e}^{(j+1)}$ with
the $(j+1)$-th unit vector $\mathbf{e}^{(j+1)}$, $1 \le j \le n-1$,
and $a_2 = \cdots = a_n = 0$ yields Lemma~\ref{le:usual-stab-pres} d).

\begin{proof}First consider the special case where $\Im(\mathbf{a}) + \relint K' \subseteq \relint K$. 
	Further assume that the 
	polynomial $g \in \C[\mathbf{z}]$ is neither zero nor stable.
	Then there exists 
	$\mathbf{w} \in \C^k$ with $\Im(\mathbf{w}) \in \R_{>0}^k$
	and $g(\mathbf{w}) = 0$, and thus 
	$f(\mathbf{a}+\sum_{j=1}^{k} w_j \mathbf{v}^{(j)}) = 0$.
	Since $\Im(\mathbf{a}) + 
	\sum_{j=1}^{k} w_j \mathbf{v}^{(j)} \in \Im(\mathbf{a}) 
	+ \relint K' \subseteq \relint K$,
	$f$ is not $K$-stable, contradiction.
	
	The general case ($\Im(\mathbf{a}) + K' \subseteq K$) follows from Hurwitz' Theorem.
\end{proof}

In the rest of this section, we present and prove 
a generalization of the 
Lieb-Sokal Lemma (Lemma \ref{le:lieb-sokal}) to conic stability.
In the usual Lieb-Sokal Lemma, we take a partial derivative of a polynomial which has degree at most $1$ in the corresponding variable. To formulate a similar result for arbitrary cones, we take a directional derivative in a direction lying in the cone, 
since
these directional derivatives preserve conic stability by Lemma \ref{le:conic-deriv}.
To this end, we need a generalized notion of degree with respect to an arbitrary direction.

\begin{definition}\label{def-dir-der}
	For $\vv\in \R^n$, we call $\rho_\vv(f)$ the 
	\emph{degree of $f$ in direction $\mathbf{v}$}, defined as the degree of the univariate polynomial $f(\w+t\vv)\in\C[t]$ for generic $\w\in\C^n$. 
\end{definition}
In particular, after taking the directional derivative in direction $\vv$ exactly $\rho_{\vv}(f)+1$ 
times, we obtain the identically zero polynomial.
The degree in the direction of a unit vector $\e^{(j)}$ coincides
with the univariate degree with respect to the variable $j$. 
We can now state the conical version of Lieb-Sokal stability preservation.

\begin{theorem}[Conic Lieb-Sokal stability preservation]\label{le:conic-lieb-sokal}
	Let $K'$ be given by $K'=K\times \R_{\geq 0}$ and $g(\z)+yf(\z)\in\C[\z,y]$ be $K'$-stable and such that $\rho_\vv(f)\leq 1$ for some $\vv\in K$. Then $g-\partial_\vv f$ is $K$-stable or $g-\partial_\vv f\equiv 0$. 
\end{theorem}

 We first establish a connection between a cone $K$ and its lift $K'$ into a higher-dimensional space, which we will use to prove Theorem~\ref{le:conic-lieb-sokal}.
 
\begin{lemma}\label{le:precond-ls}
	Let $f,g\in\C[\z]$, where $f\not\equiv 0$ and $K$-stable and let $K'=K\times \R_{\geq 0}$. Then $g+yf\in\C[\z,y]$ is $K'$-stable if and only if 
\[ 
	\Im\left(\frac{g(\z)}{f(\z)}\right)\geq 0 \quad
	\text{for all } \z\in\C^n \text{ with } \Im(\z) \in\relint K.
\]
\end{lemma}

\begin{proof}
	Let $g+yf$ be $K'$-stable. 
	Fix some $\z$ with $\Im(\z)\in\relint K$.
	 By $K$-stability, we have $f(\mathbf{z}) \neq 0$, and thus we may consider $g(\z)+yf(\z)$ as a univariate stable polynomial.

	Setting $w = -g(\z)/f(\z)$, the stability of the univariate polynomial 
	$y \mapsto g(\mathbf{z})+yf(\mathbf{z})$ implies $\Im(w) \leq 0$. It follows that 
	\begin{equation*}
	\Im\left(\frac{g(\z)}{f(\z)}\right)=\Im(-w)\geq 0.
	\end{equation*}
Conversely, suppose $\Im\left(\frac{g(\z)}{f(\z)}\right)\geq 0$ for all $\z\in\C^n$
with $\Im(\z)\in\relint K$. Assume $g\not\equiv 0$, since otherwise $yf(\z)$ would clearly be $K'$-stable. For $\z\in\C^n$ with $\Im(\z)\in\relint K$, we have for $w\in\C$ with $\Im(w)>0$ that $\frac{g(\z)}{f(\z)}\neq-w$. So $g(\z)+wf(\z)\neq 0$ and $K'$-stability follows.
\end{proof}

We can now complete the proof of Theorem~\ref{le:conic-lieb-sokal}.

\begin{proof}[Proof of Theorem~\ref{le:conic-lieb-sokal}.]
We begin by observing that $g$ is $K$-stable or $g\equiv 0$. 
Let $\vv\in K$ with $\rho_\vv(f)\leq 1$. If $\partial_\vv f\equiv 0$, there is nothing to prove. So assume $\partial_\vv f\not\equiv 0$, and thus implies $f\not\equiv0$. For a fixed $\z\in\C^n$ with $\Im(\z)\in\relint K$ we may consider $g(\z)+yf(\z)$ as a univariate polynomial in $y$. By Lemma~\ref{le:conic-deriv}, the polynomial $f(\z)=\partial_y(g(\z)+yf(\z))$ is $K$-stable. For $\z\in\C^n$ with $\Im(\z)\in\relint K$, $\vv\in K$ and $y\in\C$ with $\Im(y)>0$, we have $\Im(\z-\frac{1}{y}\vv)\in\relint K$, because
	\begin{equation*}
	\Im\left(\z-\frac{1}{y}\vv\right)=\Im(\z)-\Im\left(\frac{1}{y}\right)\vv=\Im(\z)+\frac{1}{|y|^2}\Im(y)\cdot\vv\in\relint K.
	\end{equation*} 
	It follows that $yf(\z-\frac{1}{y}\vv)$ is $K'$-stable. Since 
	$\rho_\vv(f)\leq 1$, there exist polynomials $f_0$ and $f_1$ with 
	$\rho_{f_0}(\vv),\rho_{f_1}(\vv)=0$ and
	$f(\z)=f_0(\z)+\langle\vv,\z\rangle\cdot  f_1(\z)$. Thus, the identity 
	\begin{equation*}
		yf\left(\z-\frac{1}{y}\vv\right)=yf(\z)-\partial_{\vv}f(\z)
	\end{equation*}
	implies the $K'$ stability of $yf(\z)-\partial_\vv f(\z)$.
Applying Lemma~\ref{le:precond-ls} twice gives
	\begin{equation*}
	\Im\left(\frac{g(\z)-\partial_\vv f(\z)}{f(\z)}\right)=\Im\left(\frac{g(\z) }{f(\z)}\right)+\Im\left(\frac{-\partial_\vv f(\z)}{f(\z)}\right)\geq 0.
	\end{equation*}
	Using Lemma~\ref{le:precond-ls} again, the $K'$-stability of $g(\z)-\partial_\vv f(\z)+yf(\z)$ follows. By specializing to $y=0$ and using Lemma~\ref{le:usual-stab-pres}, we obtain that $g(\z)-\partial_\vv f(\z)$ is $K$-stable or $g(\z)-\partial_\vv f(\z)\equiv 0$.
\end{proof}

Theorem~\ref{le:conic-lieb-sokal} not only generalizes the usual Lieb-Sokal Lemma to the case of arbitrary cones, but also extends it to directional derivatives with respect to every direction in the positive orthant.  We can formulate this explicitly as the following refined version for the usual stability notion.

\begin{corollary}[Refined Lieb-Sokal Lemma]\label{co:ref-lieb-sokal}
	Let $g(\z)+yf(\z)\in\C[\z,y]$ be stable and assume $\rho_\vv(f)\leq 1$ for some $\vv\in\R^n_{\ge 0}$. Then $g(\z)-\partial_{\vv}f(\z)\in\C[\z]$ is stable or
identically 0.
\end{corollary}

\section{Preservers for psd-stability\label{se:psd-preservers}}

In this section, we restrict to psd-stability.
For a complex symmetric matrix $Z\in \sym^{\C}_n$, 
we write $Z=X+iY$ with 
$X,Y\in\sym_n$.
After collecting some elementary preservers, our main results of this section
are the preservation of psd-stability under an inversion operation 
(see Theorem~\ref{th:psd-inversion} and Corollary~\ref{co:psd-inversion2}) and the preservation of psd-stability under taking initial forms with respect to positive definite matrices (see Theorem~\ref{th:initial-form-W-pd}).

For a polynomial $f \in \C[Z]$, let $f_\diag \in \C[Z]$ denote the polynomial obtained from $f$ by
substituting all off-diagonal variables by~0.
For $1 \le i \neq j \le n$, let $B_{ii}$ be the matrix which is 1 in 
entry $(i,i)$ and zero otherwise, and let
$B_{ij}$ be the matrix which is $1/2$ in entry
$(i,j)$ and $(j,i)$ and zero otherwise. Then, for a polynomial
$f = \sum_{\alpha} c_{\alpha} Z^{\alpha} \in \C[Z]$
and its equivalent version
$\tilde{f} =  \sum_{\alpha} c_{\alpha} 
  \prod_{k=1}^n z_{kk}^{\alpha_{kk}} \prod_{k < l} z_{kl}^{2 \alpha_{kl}}$
in $\C[\{z_{kl}| 1\leq k\leq l \leq n\}]$, we have the identities
$\frac{\partial f}{\partial B_{ii}} \big|_{z_{lk} := z_{kl}}
= \frac{\partial \tilde{f}}{\partial z_{ii}}$
and
$\frac{\partial f}{\partial B_{ij}} \big|_{z_{lk} := z_{kl}} 
= \frac{1}{2} \frac{\partial \tilde{f}}{\partial z_{ij}}$ as symbolic expressions.
To see this, it suffices to observe that for $i < j$ and a 
monomial
$f(Z) = z_{ij}^{\alpha_{ij}} z_{ji}^{\alpha_{ji}} \in \C[Z]$,
we have $\tilde{f} = z_{ij}^{2 \alpha_{ij}}$ and
\[
  \frac{\partial}{\partial B_{ij}} f(Z) = 
    \frac{1} {2} \alpha_{ij} z_{ij}^{\alpha_{ij}-1} z_{ji}^{\alpha_{ji}}
    + \frac{1}{2} \alpha_{ji} z_{ij}^{\alpha_{ij}} z_{ji}^{\alpha_{ji}-1}
    \in \C[Z].
\]
Substituting $z_{ji}$ by $z_{ij}$ gives 
$ \frac{\partial}{\partial B_{ij}} f(Z) \big|_{z_{ji} := z_{ij}} 
  = 
    \alpha_{ij} z_{ij}^{2 \alpha_{ij}-1} 
    \ = \ \frac{1}{2} \frac{\partial}{\partial z_{ij}} \tilde{f}.
$

\begin{lemma}[Elementary preservers for psd-stability]\label{le:psd-preserv-general}
	Let $f\in\C[Z]$ be psd-stable.
	\begin{enumerate}[label=\alph*)]
		\setlength\itemsep{0.35em}
		\item {\bf Diagonalization}: The polynomial $Z \mapsto f_\diag(Z)$ is psd-stable. 		\item {\bf Transformation}: Let $S\in\GL_n(\R)$, then $f(SZS^{-1})$ and $f(SZS^T)$ are psd-stable. 
		\item {\bf Minorization}: For $J\subseteq [n]$, let $Z_J$ be the
		symmetric $|J| \times |J|$ submatrix of $Z$ with index set $J$. Then $f(Z_J)$, the polynomial on $\sym^{\C}_{|J|}$ obtained from $f$ by setting to zero all variables with at least one index outside of $J$,
		is psd-stable or zero.
		\item {\bf Specialization}: For a fixed index $i \in [n]$, let $\hat{Z}_i$ be any
		matrix obtained from $Z$ by assigning real values to  
		$z_{ij} ,z_{ji}$ for all indices 
		$j \neq i$ and a value from $\mathcal{H}$ to 
		$z_{ii}$. Then $f(\hat{Z}_i)$, viewed as 
		polynomial on $\sym^{\C}_{n-1}$, is psd-stable
		or zero.
		\item {\bf Reduction}: For $i,j \in [n]$, let $\bar{Z}_{ij}$ be any matrix obtained from $Z$ by
		choosing real values for $z_{ik}=z_{ki}$ for $k \neq i$ and setting
		$z_{ii}:=z_{jj}$. Then $f(\bar{Z}_{ij})$,
		viewed as polynomial on $\sym^{\C}_{n-1}$,
		is psd-stable or zero.
		\item {\bf Permutation}: Let $\pi:[n]\rightarrow [n]$ be a permutation. Then $f((Z_{\pi(j),\pi(k)})_{1\leq j,k\leq n})$ is a psd-stable polynomial on $\sym^{\C}_n$.
		\item {\bf Differentiation}: $\partial_V f(Z)$ is psd-stable or zero
		for $V \in \sym_n^+$.
	\end{enumerate}
\end{lemma}

\begin{proof}

	a) Assume $f_\diag$ is not psd-stable. Then there are real symmetric
	matrices $A,B$ with $B \succ 0$ and $f_\diag(A+iB) = 0$.
	Let $A'$ and $B'$ be the matrices obtained from $A$ and $B$
	by setting all off-diagonal variable to zero. In particular, 
	$B'$ is positive definite.
	Since the only variables occurring in $f_\diag$ are the diagonal ones,
	we have $f(A'+iB') = f_\diag(A+iB) = 0$. Hence, $f$ is not
	psd-stable.

	b) Both transformations $Z \mapsto S^T Z S$ and 
	$Z \mapsto S^{-1} Z S$ preserve the inertia of $\Im(Z)$
	and thus also psd-stability. 
	
	c) 	Set $k:=|J|$ and assume without loss of generality
	$J = \{1, \ldots, k\}$. For $\varepsilon > 0$, let
	$g_{\varepsilon}$ be the polynomial on the space $\sym_k$ defined by
	$g_{\varepsilon}(Z) \ := \ f \left( \diag(Z, i \varepsilon I_{n-k}) \right)$,
	where $\diag(Z, i \varepsilon I_{n-k})$ is the block diagonal matrix with
	blocks $Z$ and $i \varepsilon I_{n-k}$.
	The psd-stability of $g$ implies the psd-stability of $g_{\varepsilon}$
	for all $\varepsilon > 0$. 
Hurwitz' Theorem ~\ref{th:hurwitz} then gives the desired result,
        because $f(Z_J) = g_0(Z)$.
	
		d) is obvious, e) and f) are similar to c), 
		and g) is the special case of Lemma~\ref{le:conic-deriv} when $K$ is the psd-cone.
\end{proof}

The diagonalization property from Lemma~\ref{le:psd-preserv-general} plays a central role in the theory of psd-stable polynomials, since it establishes connections to the usual stability notion and also gives further insights into the monomial structure of psd-stable polynomials.

\begin{corollary}\label{co:psd-reduction-property}
	Let $f \in \C[Z]$ be psd-stable. Then:
	\begin{enumerate}[label=\alph*)]
		\item The polynomial
		$(z_{11}, z_{22}, \ldots, z_{nn})\mapsto f_\diag(Z)$ is stable in $\C[z_{11},z_{22}, \ldots, z_{nn}]$.
		\item If $f(0)=0$, i.e., if $f$ does not have a constant term, then
		there is a monomial in $f$ consisting only of diagonal variables of $Z$.

		\item If $f$ is homogeneous, then
		\begin{enumerate}
			
		\item[c1)]  the sum of the coefficients of all diagonal monomials of $f$ is nonzero.
		\item[c2)] all nonzero coefficients of diagonal monomials of $f$ have the same phase.
	\end{enumerate}  
		
	\end{enumerate}
\end{corollary}
\begin{proof}
	a) By Lemma~\ref{le:psd-preserv-general}, we know that
	$f_\diag(Z)\not\equiv 0$ is psd-stable. Now it suffices
	to observe that $f_\diag (Z)\neq 0$ whenever the diagonal of $\Im(Z)$ has positive entries only.
	
	b) Let $f(0) = 0$. If each monomial in $f$ contains an off-diagonal
	variable of $Z$, then $f_\diag(Z) \equiv 0$, in contradiction to the
	psd-stability of $f_\diag(Z)$.
	
	c1) The claim follows since the sum of the coefficients of all diagonal monomials is given by $f(I_n)$ which cannot be zero due to $f(i\cdot I_n)=i^{\deg(f)}f(I_n)\neq 0$. 
	
	c2) The claim follows by combining a) with Theorem~\ref{th:same-phase}.
\end{proof}

When investigating the combinatorics of psd-stable polynomials in 
Section~\ref{se:combinatorics}, 
we will refer to the following observation, which could also 
be considered as a special case of specialization. 
	Let $f(Z)$ be psd-stable. For the real matrix variables $X$ and any fixed real matrix $B\succ0$, the polynomial $f(X+iB)$ does not have any real roots.

As the first main result in this section,
we show the following preservation statement under inversion for psd-stability.

\begin{theorem}[Psd-stability preservation 
under inversion]\label{th:psd-inversion}
	If $f(Z)\in\C[Z]$ is psd-stable, then the polynomial
	$\det(Z)^{\deg(f)}\cdot f(-Z^{-1})$ is psd-stable.
\end{theorem}

Here, the factor $\det(Z)^{\deg(f)}$ serves to ensure that the product is 
a polynomial again.
For the proof of Theorem~\ref{th:psd-inversion}, we begin with
a technical lemma.

\begin{lemma}
	\label{le:inversion}
	Let $A,B \in \sym_n$ with $B \succ 0$.
	Then the symmetric matrix $C := A+iB$ is invertible and the imaginary part 
	matrix of the symmetric matrix $C^{-1}$ is negative definite.
\end{lemma}

We will use the following elementary computation rules, which can be
verified immediately.

\begin{lemma}
\label{le:computation-rules}
Assume that $C=A+iB$ is invertible, and denote its inverse by $W = U + iV$.
\begin{enumerate}[label=\alph*)]
	\item If $A$ is invertible, then $U =(A + B A^{-1} B)^{-1}$. 
	\item If $B$ is invertible, then $V = (-B-A B^{-1} A)^{-1}$.
\end{enumerate}
\end{lemma}

We also use the following basic statement on eigenvalues in the proof of
Lemma~\ref{le:inversion}.

\begin{lemma}\label{le:alpha-in-h}
Let $A,B \in \sym_n$ and set $C = A+iB$. If $B \succ 0$ then 
$\lambda \in \mathcal{H}$ for all eigenvalues $\lambda$ of $C$.
\end{lemma}

\begin{proof}Let $B \succ 0$, and let $\lambda$ be an eigenvalue of $C$
with some corresponding eigenvector $\mathbf{v}$. Then
\begin{equation}
  \lambda = \frac{\vv^H \lambda \vv}{\vv^H \vv} 
   = \frac{\vv^H A \vv}{\vv^H \vv} 
   + i \frac{\vv^H B \vv}{\vv^H \vv}.
\end{equation}
Since $B\succ 0$, we have $\frac{\vv^H B \vv}{\vv^H \vv}>0$ and thus
$\lambda \in \mathcal{H}$.
\end{proof}

\begin{proof}[Proof of Lemma~\ref{le:inversion}.]
	Let $C = A + iB$ with $A,B \in \sym_n$ and $B \succ 0$.
	Lemma~\ref{le:alpha-in-h} gives that $C$ is invertible. 
	The symmetry of $C^{-1}$ is an immediate consequence of
	the invertibility. Indeed, $C^{-1} C = I$ implies
	$I = I^T = (C^{-1} C)^T = C^T (C^{-1})^T$. Since $C$ is symmetric,
	the matrix $(C^{-1})^T$ is the inverse of $C$, that is, $(C^{-1})^T = C^{-1}$.
	
       By Lemma~\ref{le:computation-rules}, 
	the imaginary part of 
	$W=C^{-1}$ is given by 
	$(-B - AB^{-1} A)^{-1}$.
	We observe that $B^{-1}$ is positive definite and thus
	$A B^{-1} A$ is positive semidefinite. Hence,
	$-B - A B^{-1} A$ is negative definite. Since the inverse of
	that matrix is negative definite as well, the claim follows.
\end{proof}

We can complete the proof of Theorem~\ref{th:psd-inversion}.

\begin{proof}[Proof of Theorem~\ref{th:psd-inversion}]
The inverse of a symmetric matrix $C=A+iB$ with positive definite imaginary part $B$ has a negative definite imaginary part, as shown in Lemma~\ref{le:inversion}. Thus $f(-C^{-1})\neq 0$ if $B\succ 0$. Since $\det(Z)$ is a psd-stable polynomial as well as $f$, the polynomial 
$\det(Z)^{\deg(f)}f\left(- Z^{-1}\right)$ is psd-stable.
Note that the factor $\det(Z)^{\deg(f)}$ ensures that 
$\det(Z)^{\deg(f)}f\left(- Z^{-1}\right)$ is a polynomial.
This directly follows from Cramer's rule, saying 
$Z^{-1} = \frac{1}{\det(Z)} \cdot \adj(Z)$, where $\adj(Z)$ denotes the 
adjugate matrix of $Z$.
\end{proof}

The following is a slight generalization which resembles the existing
formulation of the scalar version in Lemma~\ref{le:usual-stab-pres}.

\begin{corollary}
\label{co:psd-inversion2}
If $Z$ is a symmetric block diagonal matrix with blocks $Z_1, \ldots, Z_k$
and $f(Z) = f(Z_1, \ldots, Z_k)$ is psd-stable, then
$\det(Z_1)^{\deg_{Z_1} f} \cdot f(-Z_1^{-1}, Z_2, \ldots, Z_k)$ is a 
psd-stable polynomial. Here, $\deg_{Z_1} f$ denotes the total
degree of $f$ with respect to the variables from the block $Z_1$.
\end{corollary}

We close the section with a brief discussion and our second main result of this section on the preservation of the psd-stability of a polynomial $f\in\C[Z]$ when passing over to an initial form.  
For $f=\sum_{\alpha\in S}c_\alpha Z^\alpha \in \C[Z]$,
the initial form of $f$ is defined with respect to some functional $W$ in the 
dual space $\sym_n^*$. It is defined as
\[
\initial_W(f)= \sum_{\alpha \in S_W} c_\alpha Z^\alpha,
\]
where $S_W:=\{\alpha \in S \, : \, \langle W, \alpha\rangle_F = \max_{\beta \in S} \langle W, \beta\rangle_F \}$ and $\langle\cdot,\cdot\rangle_F$ is
the Frobenius product.
The following example shows that Theorem~\ref{th:stable-initial-form} on stability preservation under taking the initial form for any non-zero functional $\w$ does not generalize to the case of psd-stability. 

\begin{example}
	The polynomial $f\in \C[Z]$ given by 
	\[ 
	f(Z)=\det \begin{pmatrix}
		z_{11} & z_{12} & z_{13} \\
		z_{12} & z_{22} & z_{23} \\
		z_{13} & z_{23} & z_{33}
	\end{pmatrix} = z_{11}z_{22}z_{33} - z_{11}z_{23}^2 - z_{22}z_{13}^2-  z_{33}z_{12}^2+ 2 z_{12}z_{13}z_{23}.
	\]
is a psd-stable polynomial. However, taking the initial form $\initial_W(f)$ for 
	\[
	W=\begin{pmatrix}
		4 & 4 & 6 \\
		4 & 4 & 6 \\
		6 & 6 & 0
	\end{pmatrix}
	\]
	yields $\initial_W(f)=  - z_{11}z_{23}^2 - z_{22}z_{13}^2+ 2 z_{12}z_{13}z_{23}$, which vanishes at $Z= i I_3$.  Since the imaginary part of
	$i I_3$ is a positive definite matrix, $\initial_W(f)$ is not psd-stable.
\end{example}

To answer the natural question of whether psd-stability is preserved by passing over to the initial form with respect to certain symmetric matrices, we show that it is enough for $W$ to be positive definite.

For $\lambda>0$ and matrices $W\in \sym_n$, let $\lambda^W$ denote the operation given by $(\lambda^W)_{ij}:=\lambda^{w_{ij}}$. Furthermore, for two matrices $A,B\in\sym_n$ let $A\circ B$ denote the \emph{Hadamard product} of $A$ and $B$ with $(A\circ B)_{ij}=a_{ij}\cdot b_{ij}$. Generalizing the notation $|\cdot|$ for vectors,
we write $|\alpha|=\sum_{1\leq i,j\leq n} |\alpha_{ij}|$ for an exponent matrix $\alpha$.

\begin{lemma}\label{le:initial-form-psd}
	Let $f\in\C[Z]$ be psd-stable and let $W\in\sym_n$ be such that there exists some $\lambda_0 >0$ such that for every $\lambda>\lambda_0$, $\lambda^W$ is positive definite. Then $\initial_W(f)$ is psd-stable.
\end{lemma}

\begin{proof}
	The Schur product theorem states that the Hadamard product of the two positive definite matrices is positive definite. Thus we have $\lambda^W\circ A\succ 0$ for all $A\succ 0$ and $\lambda>\lambda_0$. Let $\varphi=\max\left\{
	\langle \alpha,W \rangle_F:\alpha\in \text{supp}(f)
	\right\}$ 
	and define the polynomial $f_\lambda(Z):=\frac{1}{\lambda^\varphi}f(\lambda^W\circ Z)$. This is psd-stable for any $\lambda > \lambda_0$, since the positive semi-definiteness of the imaginary part is preserved due to the previous observation. Let $\mu:=\max\{\langle \alpha, W\rangle:\alpha\in\text{supp}(f), \langle \alpha,W\rangle<\varphi\}$ and 
	$\delta:=\varphi-\mu>0$. Now, for any compact subset $C\in\sym_n$,
	\begin{align*}
		\lim\limits_{\lambda\rightarrow \infty}\underset{Z\in C}{\sup} 
		\left| f_\lambda(Z) - \initial_W(f)(Z) \right|	
		& \leq  
		\lim\limits_{\lambda\rightarrow \infty}\underset{Z\in C}{\sup} 
		\sum\nolimits_{\langle \alpha,W \rangle < \varphi} \left| \frac{1}{\lambda^{\delta}} c_\alpha Z^\alpha \right| \\ 
		& = \ 
		\lim\limits_{\lambda\rightarrow \infty}\frac{1}{\lambda^{\delta }} \underset{Z\in C}{\sup} 
		\sum\nolimits_{\langle \alpha,W \rangle < \varphi} \left|  c_\alpha Z^\alpha \right| \ = \ 0,	
	\end{align*}
	since the norm in the last equality is bounded. By Hurwitz' Theorem~\ref{th:hurwitz}, $\initial_W(f)$ is psd-stable.
\end{proof}

\begin{theorem}\label{th:initial-form-W-pd}
	Let $f\in\C[Z]$ be psd-stable and $W\in\S_n$ be positive definite, then $\initial_W(f)$ is psd-stable.
\end{theorem}
\begin{proof}
	Let $W\in\S_n$ be positive definite. Then $W^{\circ k}$, the $k$-fold Hadamard product of $W$, is positive definite for all $k \ge 1$ and so is $\exp[W]:=\sum_{k=0}^\infty \frac{W^{\circ k}}{k!}$, with the convention that $W^{\circ 0}$ is the all-ones matrix. For $\lambda>1$, we have $\ln(\lambda)\cdot W\succ 0$. Therefore,  
	\begin{equation*}
		\exp[\ln(\lambda)\cdot W]
		=(e^{w_{ij}\ln(\lambda)})_{ij}
		=(\lambda^{w_{ij}})_{ij} = \lambda^W
	\end{equation*} is positive definite. The claim now follows from Lemma~\ref{le:initial-form-psd} with $\lambda_0=1$. 
\end{proof}

\section{Combinatorics of psd-stable polynomials\label{se:combinatorics}}

This section is about combinatorial properties of the support of psd-stable polynomials,
inspired by the results
 in \cite{braenden-hpp, cos-2004,rincon-vinzant-yu}
 on the support of stable polynomials listed in 
 Sections~\ref{se:introduction} and~\ref{se:prelim}. 
 Theorem~\ref{th:struc-thm} gives
 a necessary condition on the support of any psd-stable polynomial.
 In Sections~\ref{se:binomials-non-mixed} and~\ref{sse:pol-of-det}, 
 we characterize psd-stability of binomials and non-mixed polynomials
 and the class of \emph{polynomials
 of determinants}.
 Finally, Section~\ref{sse:comb-gen-psd} discusses some aspects
 on the support of general psd-stable polynomials, provides
 a conjecture and verifies this conjecture for some special
 families of polynomials.
 We sometimes write both $z_{ij}$ and $z_{ji}$ with some $i \neq j$, but both denote the same variable $z_{ij}$ with $i \le j$, as explained at the beginning of
Section~\ref{se:psd-preservers}.

\begin{theorem}\label{th:struc-thm}
If an off-diagonal variable $z_{ij}$ (where $i<j$) occurs in a psd-stable polynomial $f\in \C[Z]$, then the corresponding diagonal variables $z_{ii}$ and $z_{jj}$ must also occur in $f$.
\end{theorem}

This mimics the basic fact about positive semidefinite matrices that if an off-diagonal entry is non-zero, the corresponding diagonal entries must also be non-zero.

\begin{proof}
We prove the contrapositive. Suppose without loss of generality that $z_{1n}$ is a variable appearing in $f$ but $z_{nn}$ is not.
We can choose an $(n-1)\times(n-1)$ complex symmetric matrix
$A$ and 
$a_{2n}, \ldots, a_{n-1,n} \in \C$
such that $\Im(A)$ is positive definite and such that substituting these values into $f$ gives a non-constant univariate polynomial $g$ in the variable $z_{1n}$. The second condition is possible because the set 
$\sym_{n-1}^{++} \times \C^{n-2}$ is an open set. Indeed, we can choose all real parts to be zero.
	
The univariate non-constant polynomial $g$ has a complex root $a_{1n}$. The assignment $z_{ij}=a_{ij}$ for all $(i,j) \neq (n,n)$ gives therefore a root of $f$ no matter what value we choose for $z_{nn}$. We now claim that
if we assign a value $a_{nn}$ with $\Im(a_{nn})$ positive and large enough, the matrix $A'$ which results from assigning these values to $Z$ has a positive definite imaginary part. 

Observe that by Sylvester's criterion of leading principle minors, it is enough to check that the determinant of $\Im(A')$ is strictly positive; the remaining leading principal minors will necessarily be positive because they are minors of $\Im(A)$, which we chose positive definite. Now, by developing the determinant along the last row, we obtain
\[
\det \Im(A') = \Im(a_{nn})\cdot \det \Im(A) + c,
\]
where $c$ is a constant. Since $\det \Im(A)$ is positive, we can choose $\Im(a_{nn})$ positive and sufficiently large so that $\det \Im(A')$ is positive.
Thus $f(A')=0$ with $\Im(A')$ positive definite, which proves that $f$ is not psd-stable.
\end{proof}
The argument used in the proof is connected to the 'positive (semi-)definite matrix completion problem', see for example \cite[Section 3.5]{bgf-1994}.
In the special case of binomials Theorem~\ref{th:struc-thm} can be extended as follows.

\begin{lemma}\label{le:struc-bin}
Let $f(Z)=c_\alpha Z^\alpha+c_\beta Z^\beta$ be a psd-stable binomial. If the two monomials $Z^\alpha$ and $Z^\beta$ do not have a common factor, then either both consist only of diagonal variables, or one only of diagonal and the other only of off-diagonal variables.
\end{lemma}
\begin{proof}
$Z^\alpha$ and $Z^\beta$ cannot both be off-diagonal monomials, since this contradicts Theorem~\ref{th:struc-thm}. It remains to be shown that neither monomial can contain both diagonal and off-diagonal variables. Suppose towards a contradiction that one of the two monomials did contain both, w.l.o.g. $Z^\beta$, and choose $j$ such that $\beta_{jj}>0$. Since the monomials of $f$ do not share any variable, $\frac{\partial Z^\alpha}{\partial z_{jj}}\equiv 0$,
where we use the derivative notation $\frac{\partial}{\partial z_{ij}}$ 
on the symmetric matrix space as introduced at the beginning of
Section~\ref{se:psd-preservers}.

Hence, $g(Z):=\frac{\partial f}{\partial z_{jj}}=c_\beta\beta_{jj}Z^{\beta'}$ is a non-zero monomial with $\beta'_{kl}=\beta_{kl}$ for $(k,l)\neq (j,j)$ and $\beta'_{jj}=\beta_{jj}-1$, that is, $g$ is a monomial containing an off-diagonal variable. Thus $g(i\cdot I_n)=0$, which is a contradiction since $g$ is psd-stable by Lemma~\ref{le:conic-deriv}.
\end{proof}

\subsection{Binomials and non-mixed 
polynomials\label{se:binomials-non-mixed}}

We give characterizations
of the support 
of psd-stable binomials. 
Some of the results will be stated
for the family of \emph{non-mixed polynomials}, which includes irreducible binomials thanks to Lemma~\ref{le:struc-bin}. 

\begin{definition}
We call a polynomial $f\in \C[Z]$ 
\emph{non-mixed} if every monomial that occurs in $f$ either consists only of diagonal variables or only of off-diagonal variables. We always write such a non-mixed polynomial as $f= \sum_{\alpha \in A} c_\alpha Z^\alpha + \sum_{\beta \in B} c_\beta Z^\beta$, where $A$ refers to the exponent matrices of
diagonal monomials and $B$ refers to the exponent matrices of off-diagonal monomials.	
\end{definition}

 It is useful to consider this larger family because it is closed under directional derivatives while the family of binomials is not. The following two theorems are the main results in this subsection.

\begin{theorem}\label{th:hom-non-mix}
	Let $f(Z)=\sum_{\alpha\in A} c_\alpha Z^\alpha + \sum_{\beta\in B} c_\beta Z^\beta$ be a homogeneous non-mixed polynomial of degree $d\geq 3$ and assume $c_\beta \neq 0$ for some $\beta \in B$. Then f is not psd-stable. 
\end{theorem}

The following theorem is a complete characterization of the 
support of psd-stable binomials, analogous to the classification of stable binomials from Theorem~\ref{th:binom-class-stable}.

\begin{theorem}\label{th:bin-class-dia}
	Every psd-stable binomial
	 is of one of the following forms:
	\begin{enumerate}[label=\alph*)]
		\item Only diagonal variables appear in $f$ and $f$ satisfies the conditions of Theorem~\ref{th:binom-class-stable}:
		 $f(Z)=Z^\gamma (c_1 Z^{\alpha_1}+c_2 Z^{\alpha_2})$ with
		 $|\alpha_1-\alpha_2|\leq 2$ and at least one of $\alpha_1,\alpha_2$ is non-zero,
		\item $f(Z)=Z^\gamma(c_1 z_{ii}z_{jj}+c_2 z_{ij}^2)$
		with $i < j$ and $\frac{c_1}{c_2}\in\R$,
	\end{enumerate}
	where $c_1, c_2 \neq 0$ and $Z^\gamma$ is a diagonal monomial.  
\end{theorem}
	This theorem shows that the only psd-stable binomials with off-diagonal variables are those described in b): in particular, at most one off-diagonal variable occurs in a psd-stable binomial, and it has degree exactly $2$.

The following lemma is a first step towards a proof of the main theorems and shows that the exponents of psd-stable binomials cannot be far apart. The proof relies on taking derivatives in direction $V^{(ij)}$, with $i \neq j$, which denotes the $n\times n$ matrix with 
$v_{ii} = v_{jj} = v_{ij} = v_{ji} = 1$ and
$0$ elsewhere. In terms of the basis matrices $B_{ij}$ introduced
at the beginning of Section~\ref{se:psd-preservers}, we have
$V^{(ij)} = B_{ii} + B_{jj} + 2 B_{ij}$.

\begin{lemma}\label{le:binom-dist}
	Let $f(Z)=c_\alpha Z^\alpha+c_\beta Z^\beta$ be a psd-stable binomial (thus 
	$c_{\alpha}, c_{\beta} \neq 0$).
	 Then $\left| \left| \alpha \right| - \left| \beta\right| \right|\leq 2$. 
\end{lemma}

\begin{proof} We may assume that the monomials of $f$ do not have a common factor since this does neither affect $|\alpha-\beta|$ nor $||\alpha|-|\beta||$. By Lemma~\ref{le:struc-bin}, either both monomials are diagonal monomials or w.l.o.g.\ only $Z^\beta$ is an off-diagonal monomial. If both monomials are diagonal, the claim follows directly from Theorem~\ref{th:braenden-js}, because 
	psd-stable polynomials involving only diagonal variables are stable polynomials.

	Now assume that $Z^\beta$ is an off-diagonal monomial. 
	Then $|\beta|\leq |\alpha|$ follows from Theorem~\ref{th:struc-thm}
	after possibly taking derivatives in direction $V^{(ij)}$ for some $z_{ij}$ appearing in $Z^\beta$. It remains to show $|\alpha|\leq |\beta|+2$. Assume to the contrary that $|\alpha|-|\beta| \geq 3$. Choose $i$ and $j$ with $i < j$ such that $z_{ij}$ occurs in $Z^\beta$. Since 
	$\frac{\partial f}{\partial V^{(ij)}}(Z)=\left(\frac{\partial}{\partial z_{ii}}+\frac{\partial}{\partial z_{jj}}\right)(c_\alpha Z^\alpha)+\frac{\partial}{\partial z_{ij}}( c_\beta Z^\beta)$ by the computation rules at the
	beginning of Section~\ref{se:psd-preservers},
	we see that $\frac{\partial f}{\partial V^{(ij)}}(Z)$ has at most two diagonal monomials, each of degree $|\alpha|-1$, and exactly one off-diagonal monomial of degree $|\beta|-1$. By applying this procedure consecutively $|\beta|$ times, we obtain a polynomial $g(Z)=\sum_{\alpha'}c_{\alpha'}Z^{\alpha'}+c_{\beta'}$, where $\sum_{\alpha'}c_{\alpha'}Z^{\alpha'}$ is a homogeneous polynomial in diagonal variables of degree $|\alpha|-|\beta|\geq 3$ and $c_{\beta'}$ is a constant. Further $g(Z)$ is psd-stable by Lemma~\ref{le:conic-deriv}. Since $g$ does not involve any off-diagonal variables, it is a stable polynomial. This is a contradiction to Theorem \ref{th:braenden-js}, since the support of $g$ does not satisfy the 
	Two-Steps Axiom. 
\end{proof}

In the following, we show that most binomials are not psd-stable by explicitly constructing a root $S$ (of the binomial or a directional derivative of it) whose imaginary part lies in the interior of the psd-cone. This root $S$ will be a symmetric $n\times n$ matrix of the form 
\begin{equation}\label{eq:S-form}
	S=\begin{pmatrix}
		s+i & t & \cdots & t\\
		t & s+i & \ddots & \vdots \\
		\vdots & \ddots & \ddots & t \\
		t & \cdots & t & s+i
	\end{pmatrix} \ \ \text{with} \ s,t\in\R.
\end{equation}
Since $\Im(S)=I_n \succ 0$, any polynomial with root $S$ is not psd-stable. 

\begin{lemma}\label{le:non-hom-bin}
	Let $f(Z)=c_\alpha Z^\alpha+c_\beta Z^\beta$ be a binomial (thus, $c_{\alpha}, c_{\beta} \neq 0$)
	with $|\alpha|>|\beta|\geq 1$ and
	such that $Z^\alpha$ and $Z^\beta$ do not have a common factor.
	Then $f$ is not psd-stable.
\end{lemma}
\begin{proof}
	Assume towards a contradiction that $f$ is psd-stable.
	Since $|\alpha| > |\beta| \ge 1$ and 
	$Z^\alpha$ and $Z^\beta$ do not share a factor, we have that $|\alpha-\beta| \geq 3$. 
	If both monomials were diagonal monomials, psd-stability would imply stability, and $|\alpha-\beta| \geq 3$ would yield a contradiction to  Theorem \ref{th:braenden-js}.
	
	Now assume that $Z^\beta$ is an off-diagonal monomial. We will show that there are $s,t\in\R$ such that $S$ is a root of $f$, thus contradicting that $f$ is psd-stable. By Lemma~\ref{le:binom-dist}, the only possibly psd-stable cases are $|\alpha|=|\beta|+1$ and $|\alpha|=|\beta|+2$.

First consider the case $|\beta| = 1$. Then $|\alpha|\in\{2,3\}$. 
After substituting $S$, $f=0$ is of the form
	\begin{equation}\label{eq:binom-S}
		(s+i)^a+bt=0 \quad \text{with} \quad a:=|\alpha|\in \{2,3\}, \ s,t\in \R, \ b\in \C\setminus\{0\}.
	\end{equation}
	One may split the real and imaginary part of equation~\eqref{eq:binom-S} to obtain two 
	real equations, denoted by (Re) and (Im). First let $a=2$. If $\Im(b)\neq 0$, there is a real solution $s=\frac{\Re(b)}{\Im(b)}+\sqrt{\left(\frac{\Re(b)}{\Im(b)}\right)^2+1}$ and $t=\frac{1-s^2}{\Im(b)}$. If $\Im(b)=0$, the solution $t=\frac{1}{\Re(b)}$ and $s=0$ may be found. Now let $a=3$. If $\Im(b)\neq 0$, (Im) implies $t=\frac{1-3s^2}{\Im(b)}$, which then gives $s^3-3s+\frac{\Re(b)}{\Im(b)}(1-3s^2)$, which has a real solution. If $\Im(b)=0$, (Im) becomes $3s^2=1$, which has the real solutions $s=\pm \frac{1}{\sqrt{3}}$. 
	Substituting these into (Re) gives a linear function in $t$, which has a real solution as well.

Now consider the case $|\beta|>1$. Choose $i$ and $j$ with $i < j$ such that the variable $z_{ij}$ occurs in $f$. Since $f$ is psd-stable, its partial derivative in direction $V^{(ij)}$ is psd-stable by Lemma~\ref{le:conic-deriv}. Further $\frac{\partial f}{\partial V^{(ij)}}(Z)=\left(\frac{\partial}{\partial z_{ii}}+\frac{\partial}{\partial z_{jj}}\right)( c_\alpha Z^\alpha)+\frac{\partial}{\partial z_{ij}}( c_\beta Z^\beta)$
	is a non-mixed polynomial with the degree of each monomial reduced by $1$ and exactly one off-diagonal monomial. Taking $|\beta|-1$ consecutive derivatives in a similar way, we obtain a non-mixed polynomial of the form $g(Z)=\sum_{\alpha'} c_{\alpha'} Z^{\alpha'} + c_{\beta'} Z^{\beta'}$ with $|\beta'|=1$ and $|\alpha'|\in\{2,3\}$. 
	Substituting $S$ into $g$ gives
	 equation~\eqref{eq:binom-S}. Thus neither $g$ nor $f$ can be psd-stable.
\end{proof}

Now we prove Theorem~\ref{th:hom-non-mix}, which shows that
homogeneous non-mixed polynomials of high degree cannot be psd-stable.

\begin{proof}[Proof of Theorem~\ref{th:hom-non-mix}]
	Assume to the contrary that $f$ is a homogeneous non-mixed psd-stable polynomial of degree at least $3$. 
	By Remark~\ref{re:real-coeff}, we can assume without loss of generality that all coefficients of $f$ are real. 
	
	First assume the degree of $f$ is $d=3$. We will show a contradiction to psd-stability by explicitly finding a forbidden root of $f$. 
	
	Let $a:=\sum_{\alpha \in A} c_\alpha$ and 
	$b:=\sum_{\beta \in B} c_\beta$. Note that $a$ and $b$ are real and 
	$a=f(I_n)\neq 0$ by Corollary~\ref{co:psd-reduction-property} c1).
	If $b\neq 0$, w.l.o.g.\ we normalize so that $a=1$.
	To obtain the desired forbidden root we look for a solution of the form $S$ introduced above, that is, real solutions $s,t$ for the equation $f(S)=(s+i)^3+bt^3=0$. By splitting the equation into real and imaginary part, we obtain the system
	\[
	\begin{array}{rrcl}
		\text{(Re)}: & \quad (s^3-3s)+bt^3&=&0, \\
		\text{(Im)}: & \quad 3s^2-1&=&0.
	\end{array}
	\]
	 Consider the positive real solution $s^*=\frac{1}{\sqrt{3}}$ of (Im). Plugging this solution into (Re) gives a real cubic in $t$, which has a real solution $t^*$. 
	
	If instead $b=0$, we tweak matrix $S$ to $S'$ as follows: let $\beta_0 \in B$ such that $c_{\beta_0} \neq 0$, and let $z_{ij}$ be a variable occurring in the monomial $Z^\beta$. Then we let $S'_{ij}=S_{ji}'=(1+\varepsilon)t$ for a small $\varepsilon>0$. The remaining entries of $S'$ are the same as those in $S$. Since $b=\sum_\beta c_\beta =0$, we have that $f(S')= (s+i)^3+ \varepsilon c_{\beta_0}  t^3$, and $\varepsilon c_{\beta_0} >0$, which means we fall into the case above with the coefficient of $t^3$ non-zero. 
	We have thus constructed solutions violating psd-stability for any such degree $3$ polynomial $f$.

	Now let $d>3$ and assume $d$ is the smallest degree such that there is a polynomial $f$ of the specified form which is psd-stable of degree $d$.
	Its partial derivative in any direction $V^{(ij)}$ 
	is psd-stable by Lemma~\ref{le:conic-deriv}. If we choose $(i,j), \ i\neq j$ such that the variable $z_{ij}$ occurs in $f$, $\frac{\partial f}{\partial V^{(ij)}}$ is a polynomial of the same form of degree $d-1$: since 
	$\frac{\partial f}{\partial V^{(ij)}}(Z)=\left(\frac{\partial}{\partial z_{ii}}+\frac{\partial}{\partial z_{jj}}\right)(\sum_\alpha c_\alpha Z^\alpha)+\frac{\partial}{\partial z_{ij}}(\sum_\beta c_\beta Z^\beta)$, 
the coefficients of off-diagonal monomials are positive multiples of those of $f$ and therefore there must be a non-zero one. This is a contradiction, since we assumed that $d$ was the smallest degree which a psd-stable polynomial of this form could have.
\end{proof}

We finally have all the tools needed to prove Theorem~\ref{th:bin-class-dia},
which provides a complete classification of the support of
psd-stable binomials.

\begin{proof}[Proof of Theorem~\ref{th:bin-class-dia}]
	Let $f$ be a binomial. Then $f$ can be written in the form $f(Z)=Z^\gamma \tilde{f}(Z)$, where $\tilde{f}(Z)=c_{\alpha} Z^\alpha+c_{\beta} Z^\beta$ is an irreducible psd-stable binomial and therefore also a non-mixed polynomial. 
	
	If all variables appearing in $f$ are diagonal variables, then $f$ is stable, and by Theorem \ref{th:braenden-js} its support has to satisfy the
Two-Steps Axiom, which leads to $|\alpha-\beta|\leq 2$ in the case of binomials. Thus, now we can assume the occurrence of an off-diagonal monomial, say 
$Z^{\beta}$. By the structure Theorem~\ref{th:struc-thm} and after possibly taking derivatives in direction $V^{(ij)}$ for some $z_{ij}$ appearing in $Z^\beta$, we see $|\beta| \le |\alpha|$.
		
	In the homogeneous case, by Theorem~\ref{th:hom-non-mix}, we have $\deg(\tilde{f})\leq 2$. The only possibility is given by 
	$\tilde{f}(Z)= c_1 z_{ii}z_{jj}+c_2 z_{ij}^2$ with $c_1, c_2 \neq 0$ and $i \neq j$, 
	since otherwise we would get a contradiction to the structure Theorem~\ref{th:struc-thm}.
		Clearly $|\alpha-\beta|\leq 2$ holds. Further we have $\frac{c_1}{c_2}\in\R$ by Remark~\ref{re:real-coeff}.
In the non-homogeneous case, i.e., $|\alpha| \neq |\beta|$,
Lemma~\ref{le:non-hom-bin} implies $\beta=0$ or $|\beta|> |\alpha|$.
$\beta=0$ is not involving an off-diagonal variable. The case 
$|\beta|>|\alpha|$ contradicts the earlier observation that
$|\beta| \le |\alpha|$. Therefore, there is no non-homogeneous 
psd-stable binomial involving off-diagonal variables.
\end{proof}

From Theorem~\ref{th:bin-class-dia}, we observe that psd-stable binomials cannot contain a monomial which is the product of different off-diagonal variables. This also holds for psd-stable homogeneous non-mixed polynomials.
 
\begin{theorem}\label{th:psd-stable-non-mixed}
	Let $f$ be a psd-stable homogeneous non-mixed polynomial of degree $2$. Then $f$ is of the form $f(Z)=\sum_{\alpha\in A}c_\alpha Z^\alpha+\sum_{i<j}c_{ij}z_{ij}^2$.
\end{theorem}

\begin{proof}

	Let $f(Z)$ be a psd-stable homogeneous non-mixed polynomial of degree $2$ and assume to the contrary that there is a monomial $z_{ij}z_{kl}$ in $f$ involving distinct variables, that is, $\{i,j\} \neq \{k,l\}$. Note that the index sets $\{i,j\}$ and $\{k,l\}$ can intersect.
	The order of the variable matrix must therefore be at least~3.
	
	Consider $S'$ as a modified version of $S$ from~\eqref{eq:S-form} with $S'_{ij}:=S'_{ji}:=t_1, \ S'_{kl}:=S'_{lk} := t_2$ for
	complex $t_1$ and $t_2$ and set all other off-diagonal entries of $S'$ to $0$ while the diagonal of $S$ is set to some complex value $s$. 
	Thus, up to a factor, $f(S')=0$ is of the form
	\begin{equation}\label{eq:real-sol-S}
	s^2+c_1 t_1^2+c t_1 t_2 + c_2 t_2^2=0
	\end{equation}
	with some constants $c_1,c_2, c$ and $c \neq 0$. 
	Since $f$ is hyperbolic due to its homogeneity and psd-stability, we may assume $c_1,c_2,c$ to be real.
	
	Since $f$ is hyperbolic, the quadratic polynomial $g$ in $s,t_1,t_2$
	on the left hand side of~\eqref{eq:real-sol-S} is hyperbolic 
	as well. Hyperbolic quadratic polynomials have signature $(n-1,1)$
	or $(1,n-1)$ (\cite{garding-59}, see, e.g., also \cite{pemantle-2012}). 
	Since the term $s^2$ in $g$ comes from
	a substitution into the terms $z_{11}^2, \ldots, z_{nn}^2$,
	the representation matrix of $g$ must have signature $(2,1)$. 
	Hence, the lower right $2 \times 2$-matrix of the representation matrix
	\[
	\begin{pmatrix}
	1 & 0 & 0 \\
	0 & c_1 & \frac{c}{2} \\
	0 & \frac{c}{2} & c_2
	\end{pmatrix}
	\]
	has signature $(1,1)$.
	If at least one of the $c_i$ is positive, then we
	can choose real values for $t_1$ and $t_2$ such that
	$s^2 + \gamma = 0$ with some $\gamma > 0$,
	which gives among the two solutions for $s$ one with
	positive imaginary part. 
	If one of the $c_i$, say, $c_2$, is zero and
	$c_1 \le 0$, then setting $t_1 = 1$ and $t_2 = \frac{1-c_1}{c}$
	gives the solution $s=i$ with positive imaginary part.
	It remains to consider the case $c_1 < 0$, $c_2 < 0$, in which
	the signature condition implies $(c/2)^2 > c_1 c_2$. By choosing $t_1, t_2$ to 
	satisfy $t_1^2 = - \frac{1}{c_1}$, $t_2^2 = - \frac{1}{c_2}$, we obtain
	\[
	t_1^2 t_2^2 \left( \frac{c}{2} \right)^2 \ > \
	t_1^2 c_1 t_2^2 c_2 
	\ = \ \frac{1}{4} (t_1^2 c_1 + t_2^2 c_2)^2,
	\]
	which can formally be viewed as the equality case
	of the arithmetic-geometric inequality. We can pick the signs
	of $t_1, t_2$ such that $c t_1 t_2 > 0$. And the previous inequality implies
	\[
	| c t_1 t_2 | > | c_1 t_1^2 + c_2 t_2^2|
	\]
	(and the expression in the argument of the absolute value on the right hand side is negative).
	Hence, we obtain $s^2 + \gamma = 0$ for some positive $\gamma$,
	which gives among the two solutions for
	$s$ one with positive imaginary part.
	
	Altogether, we have constructed a zero $S'$ of $f$ with $\Im(S') \succ 0$,
	which contradicts the psd-stability of $f$.
\end{proof}

\subsection{Polynomials of determinants}\label{sse:pol-of-det}

We show that the following class of polynomials of determinants satisfies 
a generalized jump system criterion with regard to psd-stability.
	Suppose that the symmetric matrix of variables $Z$ is a diagonal block matrix with blocks $Z_1, \dots, Z_k$. 
	A \emph{polynomial of determinants} is a polynomial in $Z$ of the form $f(Z_1, \ldots, Z_k)=\sum_{\alpha} c_\alpha\det(Z)^\alpha$, where we define $\det(Z)^\alpha=\det(Z_1)^{\alpha_1}\cdots \det(Z_k)^{\alpha_k}$.

We say a polynomial of determinants $f(Z_1, \ldots, Z_k)=\sum_{\alpha} \det(Z)^\alpha$ is written in \emph{standard form} if the largest possible determinantal monomial is factored out, i.e., 
$f(Z_1, \ldots, Z_k)= $ $\det(Z)^\gamma\sum_\beta c_{\beta} \det(Z)^\beta$ $=\det(Z)^\gamma\tilde{f}(Z)$, and all $c_\beta \neq 0$.  
We investigate the following notion of support for polynomials of determinants. 

\begin{definition}\label{def:determinantal-supp}
	Let $f(Z_1, \ldots, Z_k)=\sum_{\alpha}c_\alpha\det(Z)^\alpha$ be a polynomial of determinants. Then the \emph{determinantal support} is defined as $\supp_{\det} (f)=\{\alpha\in\Z_{\geq 0}^k:c_\alpha\neq 0 \}$.
\end{definition}
Note that the determinantal support specializes to the usual support when $Z$ is a diagonal matrix, that is, all $Z_i$ are $1\times 1$ matrices of a single variable.
As a corollary of Theorem~\ref{th:braenden-js}, we obtain the following analogue for the determinantal support of psd-stable polynomials of determinants.
\begin{corollary}\label{th:det-jump-system}
	Let $f(Z_1, \ldots, Z_k)=\sum_{\alpha}c_\alpha\det(Z)^\alpha$  be psd-stable. Then the determinantal support of $f$ forms a jump system.
\end{corollary}

The next theorem shows that psd-stable polynomials of determinants have a very special structure.

\begin{theorem}\label{th:mat-size-poly-of-det}
	Let $f(Z_1, \ldots, Z_k) = \det(Z)^\gamma \sum_{\beta \in B} c_\beta \det(Z)^\beta = \det(Z)^\gamma \tilde{f}(Z)$ be a psd-stable polynomial of determinants in standard form. Then any block $Z_i$ appearing in $\tilde{f}$ (that is, any $Z_i$ such that there is $\beta \in B$ with $\beta_i >0$) has size $d_i \leq 2$.  
	
	Further, for any matrix $Z_i$ which has size exactly $2$, let $C_i=\max_{\beta \in B} \beta_i$. Then if $\beta \in B$, then also $\beta+ c \e_i \in B$ for all  $- \beta_i \leq c \leq C_i -\beta_i$.
\end{theorem}

\begin{proof}
	Observe that, by construction, a variable in the matrix $Z_i$ does not appear in any other matrix $Z_j$. This ensures that all vectors in the support of the polynomial $\tilde{f}_\diag(Z)$, which involves only the diagonal variables, are of the form 
	\begin{equation}\label{eq:supp_diag}
		(\underbrace{\beta_1, \dots, \beta_1}_{d_1 \text{ times}},  \dots, \underbrace{\beta_k, \dots, \beta_k}_{d_k \text{ times}}),
	\end{equation}
	where $\beta=(\beta_1,  \dots, \beta_k) \in B$ is an exponent vector of $\det(Z)$ in $\tilde{f}$ and $d_i$ is the size of the matrix $Z_i$ for each $i$.
	
	Further, $\tilde{f}_\diag$ is stable and its support is therefore a jump system. Suppose now that some matrix $Z_i$, say $Z_1$, has size $d_1\geq 3$.
 Since $f$ is in standard form, there are $\beta \in B$ such that $\beta_1 >0$ and $\beta' \in B$ such that $\beta'_1=0$. Then there are corresponding vectors $\alpha = (\underbrace{\beta_1, \dots, \beta_1}_{d_1 \text{ times}}, \beta_2, \dots)$ and $\alpha' = (\underbrace{0, \dots, 0}_{d_1 \text{ times}}, \beta_2', \dots)$ in the support of $\tilde{f}_\diag$, which is a jump system. Thus $\e_1 $ is a valid step from $\alpha'$ to $\alpha$, but since $\alpha'+\e_1=(1, 0, \dots, 0, \dots)$ is not of the form (\ref{eq:supp_diag}) it cannot belong to the support of $\tilde{f}_\diag$. Now by definition of a jump system, there must be a step from $\alpha' + \e_1$ to $\alpha$ which is in the support. However, whichever step we take will lead us again to a vector where the first $d_1$ entries are not all equal, since $d_1\geq 3$, and thus none of these vectors can be in the support of $\tilde{f}_\diag$, contradicting the fact that it is a jump system. Thus all blocks $Z_i$ in $\tilde{f}$ must have size $d_i \leq 2$.
	
	Now suppose that $d_i=2$ for some block $Z_i$, without loss of generality let it be $Z_1$. Just as before, we know there are  $\beta \in B$ such that $\beta_1 >0$ and $\beta' \in B$ such that $\beta'_1=0$; further,  If $C_1=\max_{\beta \in B} \beta_1$, then there is also a vector $\beta'' \in B$ such that $\beta''_1=C_1$. This implies that in the support of $\tilde{f}_\diag$ there are vectors $\alpha = (\beta_1, \beta_1,  \dots)$, $\alpha' = (0, 0,  \dots)$ and $\alpha'' = (C_1, C_1,  \dots)$. Thus $\alpha - \e_1= (\beta_1 - 1, \beta_1, \dots)$ is a valid step from $\alpha$ to $\alpha'$. Just as before, $\alpha - \e_1$ does not belong to the support of $\tilde{f}_\diag$ because it is not of the form of (\ref{eq:supp_diag}). Thus there must be a further step from $\alpha - \e_1$ towards $\alpha'$ which is in the support. The only such step is in the second coordinate, so that (\ref{eq:supp_diag}) is satisfied, and thus $\alpha -\e_1 -\e_2 \in \supp(\tilde{f}_\diag)$. This argument can be repeated until we obtain the statement of the theorem. 
\end{proof}

\subsection{Considerations on the support of general psd-stable polynomials}\label{sse:comb-gen-psd}

By Theorem~\ref{th:braenden-js}, the support of a stable polynomial 
defines a jump system. Hence, there cannot be large gaps in the support, that is, if two vectors are in the support and are far apart, there is some other vector of the support between them.
The families studied in 
Subsections~\ref{se:binomials-non-mixed} and~\ref{sse:pol-of-det} 
suggest that a similar phenomenon happens for psd-stability: when there are too-large gaps in the support, the polynomial cannot be psd-stable. 

In order to quantify what a \emph{large gap} should be, we make two observations. First, since restricting a psd-stable polynomial in the symmetric matrix variables $Z$ to its diagonal yields a stable polynomial, between two monomials involving only diagonal variables the Two-Steps Axiom holds. A weaker statement is that between any two such monomials there is a sequence of linear and double steps which does not leave the support of the polynomial, where we define a \emph{linear step} from a monomial to be multiplying the monomial by $z_{ij}^{\pm 1}$, a \emph{double step} multiplying by $z_{ij}^{\pm 1}z_{kl}^{\pm 1}$.

Recall from Lemma~\ref{le:determinant} that a prominent example of psd-stable polynomials is the symmetric determinant $\det(Z)$. 
In the symmetric matrix variables $(z_{ij})_{i \le j}$, 
its support has a special structure: it contains all monomials that can be obtained from 
$z_{11} \cdots z_{nn}$ by transpositions of indices, that is, by successively multiplying the monomial by $z_{ij}z_{kl}z_{ik}^{-1}z_{jl}^{-1}$ for some indices $i, j, k, l \in [n]$. We call such a move on monomials a \emph{transposition step}. 

\begin{lemma}\label{le:det-transp}
	Any two monomials in the support of the symmetric determinant $\det(Z)$ are linked by a sequence of transposition steps decreasing the distance between the monomials which never leave the support.
\end{lemma}
\begin{proof}
	Monomials in $\det(Z)$ are precisely those products of symmetric
	variables $z_{ij}$ (where $i \le j$)
	such that each index $k \in [n]$ appears exactly twice. Indeed,
	when considering the determinant as a polynomial in $n^2$ (i.e.,
	non-symmetric) variables, each monomial corresponds to a permutation in the symmetric group
	$S_n$, and thus each element of $[n]$ must appear precisely once in the
	rows and once in the columns index in the monomial. When considering the
	determinant as a polynomial in the symmetric variables, certain distinct
	permutations define the same monomial. Observe that the variable
	$z_{ij}$ appears in the monomial defined by a permutation $\pi$ if
	either $i= \pi(j)$ or $j=\pi(i)$. Thus, both a cycle $\sigma=(i_1 i_2
	\dots i_k) \in S_k$ and its inverse $(i_1 i_k i_{k-1} \dots i_2 ) $
	yield the monomial $\Pi_{j} z_{i_j i_{j+1}}$, and in general, two
	permutations correspond to the same monomial if and only if their cycle
	decompositions are made of pairwise the same or inverse cycles. Since two
	such permutations have the same sign, there is no cancelation of
	monomials in the symmetric determinant $\det(Z)$.
	
	Thus, applying any transposition step to any monomial of $\det(Z)$ will
	yield another monomial of $\det(Z)$: exchanging $z_{ij}z_{kl}$ with
	$z_{ik}z_{jl}$ or $z_{il}z_{kj}$ preserves the property that each index
	appears exactly twice.	We now only need to show that, given any two monomials $Z^\alpha$ and $Z^\beta$ of $\det(Z)$, 
	there exists a transposition step
	from $\alpha$ to $\beta$. Choose a variable
	$z_{ij}$ such that $z_{ij} \mid Z^\alpha$ but 
	$z_{ij} \nmid Z^\beta$. There must
	be an index $k \neq j$ such that $z_{ik} \mid Z^\beta$ and an index $l \neq
	i$ such that $z_{jl} \mid Z^\beta$. Then multiplying $Z^\alpha$ by
	$z_{ij}^{-1}z_{kl}^{-1}z_{ik}z_{jl}$ is a transposition step, since it
	decreases the distance to $\beta$ in the norm $| \cdot |$.
\end{proof} 

We conjecture that a property inspired by the structure of the determinant and that of stable polynomials holds for all psd-stable polynomials. 

\begin{conjecture}\label{conj:support}
	For any monomial $Z^\beta$ appearing in a psd-stable polynomial, there is a diagonal monomial $Z^\alpha$ appearing in $f$ which can be reached by a sequence of linear, double and transposition steps which decrease the distance from $\beta$ to $\alpha$ and which never leave the support of $f$.
\end{conjecture}

\begin{example}
	The polynomial
	\begin{align*} f(Z)&=(z_{11}+z_{22}-2z_{12})(z_{11}z_{33}-z_{13}^2)\\
	&= z_{11}^2z_{33} + z_{11}z_{22}z_{33}-2z_{11}z_{33}z_{12}-z_{11}z_{13}^2 - z_{13}^2z_{22}+2z_{12}z_{13}^2
	\end{align*}
	is psd-stable because it is the product of two psd-stable polynomials: the first one 
is the derivative in direction $V^{(12)}$ of the $2\times 2$ determinant, 
the other one is a $2\times 2$ determinant sharing one variable with the first.  
	
	This polynomial satisfies Conjecture \ref{conj:support}: for example, if we choose the monomial $z_{12}z_{13}^2$, with a double step we reach $z_{11}z_{13}^2$, which is also in the support of $f$, and with a transposition step we reach $z_{11}^2z_{33}$, a diagonal monomial in the support. Notice that the double step produces a monomial whose exponent vector is closer to the exponent of the final diagonal monomial (with respect to $| \cdot |$). Such a sequence of valid steps can be found for all monomials of $f$.
\end{example}

As evidence for the conjecture, we observe that it holds for the classes of polynomials we have studied. 
\begin{lemma}
	Psd-stable binomials satisfy Conjecture \ref{conj:support}.
\end{lemma}
\begin{proof}
By Theorem~\ref{th:bin-class-dia}, $c_\alpha z_{ii}z_{jj}+ c_\beta z_{12}^2$ is the only irreducible psd-stable binomial involving off-diagonal variables. Clearly, it is exactly one transposition step between the both monomials.
\end{proof}
\begin{lemma}
	Psd-stable homogeneous non-mixed polynomials satisfy Conjecture~\ref{conj:support}.
\end{lemma}
\begin{proof} Let $f$ be a psd-stable homogeneous non-mixed polynomial. If $f$ does not involve off-diagonal monomials, the claim follows from the jump system property of usual stable polynomials. Thus assume that $f$
involves off-diagonal variables. We have $d:=\deg(f)\leq 2$ by Theorem~\ref{th:hom-non-mix}. In the case of $d=1$ there is a double step between every two monomials of $f$, thus assume $d=2$ and let $c_\beta Z^\beta$ be an off-diagonal monomial of $f$. By Theorem~\ref{th:psd-stable-non-mixed}, $c_\beta Z^\beta$ is of the form $c_{jk}z_{jk}^2$ for some $j\neq k$. Let $J=\{j,k\}$, then $f(Z_J)$ is psd-stable by Lemma~\ref{le:psd-preserv-general} c). By the structure Theorem~\ref{th:struc-thm}, $f(Z_J)$ is of the form
	\begin{equation*}
		f(Z_J)=c_1z_{jj}^2+c_2z_{jj}z_{kk}+c_3z_{kk}^2+c_{jk}z_{jk}^2 
	\end{equation*}
	with $c_k\in\C$ such that $z_{jj}$ and $z_{kk}$ both appear. We claim that $c_2 \neq 0$.
	Assuming $c_2=0$ gives $c_1,c_3\neq 0$. Reducing $f(Z_J)$ to the diagonal contradicts the jump system property and thus we obtain $c_2\neq 0$. Therefore, the monomial $c_2z_{jj}z_{kk}$ appears in $f(Z_J)$ and hence also in $f(Z)$. Thus, it is a transposition step from $c_{jk}z_{jk}^2$ to the corresponding diagonal monomial $c_2z_{jj}z_{kk}$. 
\end{proof}
\begin{lemma}
	Psd-stable polynomials of determinants satisfy Conjecture \ref{conj:support}.
\end{lemma}
\begin{proof}
	Every monomial $Z^\beta$ in a polynomial of determinants $f$ belongs to a determinantal monomial $\det(Z)^\gamma$ and thus is a product of monomials $Z^{\beta_j}$ (with multiplicities $\gamma_j$) belonging to determinantal blocks $\det(Z_j), \ 1\leq j \leq k$. Let $Z^{\alpha_j}$ be the diagonal monomial of block $\det(Z_j)$. By Lemma~\ref{le:det-transp} there is a sequence of transposition steps from $Z^{\beta_j}$ to $Z^{\alpha_j}$ which never leaves the support of $\det(Z_j)$ for all $1\leq j \leq k$. Concatenation of these sequences (with multiplicities $\gamma_j$) gives a sequence of transposition steps from $Z^\beta$ to the diagonal monomial $Z^\alpha$ of $\det(Z)^\gamma$ which never leaves the support of $f$. 
\end{proof}

Another class of psd-stable polynomials which satisfy Conjecture~\ref{conj:support} are the psd-stable lpm polynomials
introduced in \cite{blek-et-al-2021}, which are
polynomials of the form $f(Z)=\sum_{J\subseteq [n]} c_J\det(Z_J)$, where $Z_J$ is the square submatrix of $Z$ with index set $J$. Indeed,  every monomial belongs to a square minor of $Z$, and since every minor has a different index set, there is no cancellation of monomials in the sum. Thus for each summand the Lemma \ref{le:det-transp} holds and it holds for the whole polynomial as well.

\bibliography{bibstable}
\bibliographystyle{plain}

\end{document}